\documentclass[12pt]{amsart}

\usepackage{graphicx} 
\usepackage{scrextend}
\usepackage{amsfonts, amsmath, amssymb,amsthm,comment}  
\usepackage{times,enumerate}
\usepackage{mathdots}%
\usepackage{nccmath}
\usepackage{mathtools}
\usepackage{multicol}
    {\end{pmatrix}\end{medsize}}%
\usepackage{dsfont,bbm}
\usepackage{rotating}
\usepackage{lscape}
\usepackage{url}
\usepackage{multicol}
\usepackage{thmtools, thm-restate}
\usepackage{blkarray}
\usepackage{multicol}
\usepackage[margin=2.5cm]{geometry}
 \usepackage{hyperref}
\usepackage{cancel}
\usepackage{kbordermatrix}
\renewcommand{\kbldelim}{(}
\renewcommand{\kbrdelim}{)}
\renewcommand{\kbrowstyle}{\displaystyle}
\renewcommand{\kbcolstyle}{\displaystyle}
\def\VR{\kern-\arraycolsep\strut\vrule &\kern-\arraycolsep}
\def\vr{\kern-\arraycolsep & \kern-\arraycolsep}
\makeatletter
\newcommand*{\sublabel}[1]{%
    \let\old@currentlabel\@currentlabel%
    \renewcommand{\@currentlabel}{\theenumii}%
    \label{#1}%
    \let\@currentlabel\old@currentlabel%
}
\makeatother

\newcommand{\block}[1]{
  \underbrace{\begin{matrix}0 & \cdots & 0\end{matrix}}_{#1}
}

\usepackage[usenames, dvipsnames]{xcolor}
\DeclareGraphicsExtensions{.pdf,.png,.jpg}

\hypersetup{
    colorlinks,
    linkcolor={blue},
    citecolor={blue},
    urlcolor={blue}
}

\DeclareMathOperator{\Span}{span}

\DeclareMathOperator{\Rank}{rank}

\makeatletter
\def\widebreve{\mathpalette\wide@breve}
\def\wide@breve#1#2{\sbox\z@{$#1#2$}%
     \mathop{\vbox{\m@th\ialign{##\crcr
\kern0.08em\brevefill#1{0.8\wd\z@}\crcr\noalign{\nointerlineskip}%
                    $\hss#1#2\hss$\crcr}}}\limits}
\def\brevefill#1#2{$\m@th\sbox\tw@{$#1($}%
  \hss\resizebox{#2}{\wd\tw@}{\rotatebox[origin=c]{90}{\upshape(}}\hss$}
\makeatletter

\newcommand{\x}{{\tt x}}
\newcommand{\y}{{\tt y}}
\newcommand{\RR}{\mathbb R}

\newcommand{\NN}{\mathbb N}

\newcommand{\Tt}{{\tt t}}

\newcommand{\cH}{\mathcal H}

\newcommand{\cZ}{\mathcal Z}

\newcommand{\benu}{\begin{enumerate}}
\newcommand{\eenu}{\end{enumerate}}
\newcommand{\bop}{\begin{opomba}}
\newcommand{\eop}{\end{opomba}}

\newtheorem{theorem}{Theorem}[section]

\newtheorem{lemma}[theorem]{Lemma}

\newtheorem{problem}{Problem}

\theoremstyle{definition}

\newtheorem{example}[theorem]{Example}

\newcommand{\mbf}{\mathbf}

\theoremstyle{remark}
\newtheorem{remark}[theorem]{Remark}

\numberwithin{equation}{section}

\begin{document}

\title{Gaussian Quadratures with prescribed nodes via moment theory}

\author[R. Nailwal]{Rajkamal Nailwal}
\address{Rajkamal Nailwal, Institute of Mathematics, Physics and Mechanics, Ljubljana, Slovenia.}
\email{rajkamal.nailwal@imfm.si, raj1994nailwal@gmail.com}

\author[A. Zalar]{Alja\v z Zalar}
\address{Alja\v z Zalar, 
Faculty of Computer and Information Science, University of Ljubljana  \& 
Faculty of Mathematics and Physics, University of Ljubljana  \&
Institute of Mathematics, Physics and Mechanics, Ljubljana, Slovenia.}
\email{aljaz.zalar@fri.uni-lj.si}
\thanks{The second-named author was supported by the ARIS (Slovenian Research and Innovation Agency)
research core funding No.\ P1-0228 and grants No.\ J1-50002, J1-60011.} 



%

\subjclass[2020]{Primary 65D32, 47A57, 47A20, 44A60; Secondary 
15A04, 47N40.}

\keywords{Gaussian quadrature, truncated moment problem, representing measure, moment matrix, localizing moment matrix.}
\date{\today}
\maketitle

\begin{abstract}
    Let $\mu$ be a positive Borel measure on the real line
    and let $L$ be the linear functional on univariate polynomials of bounded degree, defined as integration with respect to $\mu$.
	In \cite{BKRSV20}, the characterization of all minimal quadrature rules of $\mu$ in terms of the roots of a bivariate polynomial is given and two determinantal representations of this polynomial are established.
    In particular, the authors solved the question of the existence of a minimal quadrature rule with one prescribed node, leaving open the extension to more prescribed nodes \cite[Problem 1]{BKRSV20}. In this paper, we solve this problem using moment theory as the main tool.
\end{abstract}

\section{Introduction}

Given a real sequence 
\begin{equation}
	\label{230422-1204}
		\gamma\equiv \gamma^{(D)}=(\gamma_{0},\gamma_{1},\ldots,\gamma_{D})\in \RR^{D+1}
\end{equation}
of degree $D$, $D\in \NN\cup \{0\}$,
the \textbf{$\RR$--truncated moment problem ($\RR$--TMP)} for $\gamma$
asks to characterize the existence of a positive Borel measure $\mu$ on $\RR$, such that
	\begin{equation}
		\label{moment-measure-cond-univ}
			\gamma_{i}=\int_{\RR}x^i d\mu\quad \text{for}\quad i\in \NN\cup\{0\},\;0\leq i\leq D.
	\end{equation}
If such a measure exists, we say that $\gamma$ has a \textbf{$\RR$--representing measure ($\RR$--rm)}.  

The solution to the $\RR$--TMP in terms of the properties of the corresponding Hankel matrix, called \textbf{moment matrix},
was established in \cite{CF91} relying on certificates of nonnegativity for polynomials, while special cases were already solved in \cite{AK62,KN77,Ioh82}. See also \cite[Chapter 9]{Sch17} for a nice overview.
In \cite{CF91} the question of the existence and uniqueness of the solution is settled. In the nonsingular case, i.e., when the moment matrix is positive definite, there is a one-parametric family of finitely atomic representing measures, supported on the smallest possible number of atoms equal to the rank of the moment matrix. Such a solution is called \textbf{Gaussian quadrature rule (GQR)} in numerical analysis and is important because it allows efficient computation of definite integrals of univariate functions.
In \cite{BKRSV20}, the authors investigated the question of characterizing the GQR\textit{s} in terms of symmetric determinantal representations involving moment matrices. Their main result \cite[Theorem 1.4]{BKRSV20}, obtained using convex analysis and algebraic geometry, characterizes when a given atom, also called a \textbf{node}, is a part of a GQR and how to determine other nodes in terms of the determinant of a certain univariate matrix polynomial. Here we emphasize that they allow node at infinity, called \textbf{evaluation at $\infty$}, and the corresponding quadrature rule is called \textbf{generalized GQR (gGQR)}. A characterization for the case that all nodes are real, in terms of the invertibility of a certain matrix, is also given. A characterization of (g)GQR\textit{s} with finitely many prescribed nodes remains an open problem:

\begin{problem}[{\cite[Problem 1]{BKRSV20}}]
\label{problem}
Given $d_1,d_2\geq 1$, a degree $D:=d_1+2d_2-1$ sequence
$\gamma^{(D)}$ as in \eqref{230422-1204}
such that the moment matrix
$M_{\lfloor\frac{D}{2}\rfloor}$ is positive definite and real numbers $x_{1},\ldots x_{d_1}$, when does a (g)GQR with $d_1+d_2$ nodes,
containing $x_{1},\ldots,x_{d_1}$, exist?
\end{problem}

In the formulation of Problem \ref{problem} from \cite{BKRSV20}, $d_1=n-1$, $d_2=\ell+1$ and the sequence $\gamma^{(D)}$ is assumed to come from some positive measure. Since $M_{\lfloor \frac{D}{2}\rfloor}$ is assumed to be positive definite, the same automatically holds in the formulation above (see e.g., Theorem \ref{Hamburger} below).
The choice of degree $D$ in Problem \ref{problem} is natural due to the fact that for each prescribed node only its density is unknown, while for each non-prescribed node both the value of the node and its density are unknown. In total, there are exactly $D+1$ variables, which is equal to the length of $\gamma^{(D)}$.
\bigskip

A necessary condition for the solution to Problem \ref{problem} is \cite[Proposition 4.1]{BKRSV20}, but it is not sufficient for $d_1>1$ 
as demonstrated by \cite[Example 4.2]{BKRSV20}. 
Finding other necessary conditions to settle the problem was the main motivation for this paper.

The main result of this paper is the solution to Problem \ref{problem}, obtained using moment theory.
 Section \ref{sec:preliminaries} introduces the notation and tools needed to establish the main results.
The solution to the GQR case of Problem \ref{problem} is
Theorem \ref{thm:main},
while the gGQR case is covered by
Theorem \ref{the:main-gene}. The main technique to establish these results is to observe the appropriate univariate sequence of $\gamma$, associated with the monic polynomial that has zeros exactly in the prescribed nodes. This sequence uniquely determines both the remaining nodes in the potential representing measure and the extension of the original sequence. Then the solution follows by applying the solution to the $\RR$--TMP from \cite{CF91} to the extended sequence. In Examples \ref{ex:no-measure} and \ref{ex:measure} we demonstrate the application of Theorems \ref{thm:main} and \ref{the:main-gene} on numerical examples.

\section{Preliminaries}
\label{sec:preliminaries}

We write $\RR^{n\times m}$ for the set of $n\times m$ real matrices. 
For a polynomial $f\in \RR[\x]$
we denote by $\cZ(f):=\{x\in \RR\colon f(x)=0\}$ its set of zeros.

\subsection{Moment matrix}
Let $D=2d$, $d\in \NN$.
For $\ell\in \NN$, $\ell\leq d$
the Hankel matrix of size $(\ell+1)\times (\ell+1)$ corresponding to the sequence $\gamma$ as in \eqref{230422-1204},
with columns and rows indexed by the monomials $1,X,\ldots,X^\ell$,
is defined by
	\begin{equation}\label{vector-v}
		M_{\ell}:=\left(\gamma_{i+j-2} \right)_{i,j=1}^{\ell+1}
					=	\kbordermatrix{
							& \textit{1} & X & X^2 & \cdots  & X^{\ell} \\
							\textit{1} & \gamma_0 & \gamma_1 & \gamma_2 & \cdots &\gamma_\ell\\
							X & \gamma_1 & \gamma_2 & \iddots & \iddots & \gamma_{\ell+1}\\
							X^2 & \gamma_2 & \iddots & \iddots & \iddots & \vdots\\
							\vdots &\vdots 	& \iddots & \iddots & \iddots & \gamma_{2\ell-1}\\
							X^\ell & \gamma_\ell & \gamma_{\ell+1} & \cdots & \gamma_{2\ell-1} & \gamma_{2\ell}
						}
	\end{equation}
and is called the \textbf{$\ell$--th truncated moment matrix} of $\gamma$.
For $i,j\in \NN\cup\{0\}$, $i\leq d$, $j\leq d$, let 
\begin{equation}
    \label{bold-v}
    \mbf{v}_i^{(j)}:=\left( \gamma_{i+r-1} \right)_{1\leq r\leq j+1}\in \RR^{j+1}.
\end{equation}
Using this notation, we have that
	$$M_d=\left(\begin{array}{ccc} 
				\mbf{v}_0^{(d)} & \cdots & \mbf{v}_d^{(d)}
			\end{array}\right).$$

For $p(x)=\sum_{i=0}^{\ell} a_{i}x^i\in \RR[x]$
we define the \textbf{evaluation} $p(X)$ on the columns of the matrix $M_\ell$ by $p(X)=a_0\textit{1}+\sum_{i=1}^{\ell} a_i X^i$, where $\textit{1}$ and $X^i$ represent the columns of $M_\ell$ indexed by these monomials.
Then $p(X)$ is a vector from the linear span of the columns of $M_\ell$. 
If this vector is the zero one,
then we say $p$ is a \textbf{column relation} of $M_\ell$.

As in \cite{CF91}, the \textbf{rank} of $\gamma$, denoted by $\Rank \gamma$, is defined by
	$$\Rank \gamma=
	\left\{\begin{array}{rl} 
		d+1,&	\text{if } M_d \text{ is nonsingular},\\[0.5em]
		\min\left\{i\colon \mbf{v}_i^{(d)}\in \Span\{\mbf{v}_0^{(d)},\ldots,\mbf{v}_{i-1}^{(d)}\}\right\},&		\text{if } M_d \text{ is singular}.
	 \end{array}\right.$$
If $\Rank \gamma<d+1$, we say that $\gamma$ is \textbf{singular}. Else $\gamma$ is \textbf{nonsingular}.
We call $\gamma$ \textbf{positively recursively generated (prg)} if for  $r=\Rank m$ the following two conditions hold:
\begin{enumerate}
	\item $M_{r-1}$ is positive definite.
	\item\label{def:coef-gen-poly} 
    If $r<d+1$, denoting 
    $(\varphi_0,\ldots,\varphi_{r-1}):=M_{r-1}^{-1} \mbf{v_r^{(r-1)}}$,
		the equality
		\begin{equation}\label{recursive-generation}
			  \gamma_j=\varphi_0\gamma_{j-r}+\cdots+\varphi_{r-1}\gamma_{j-1}
		\end{equation}
		holds for $j=r,\ldots,2d$.
\end{enumerate}
If $\gamma$ is singular and prg, we call 
\begin{equation}
	\label{230422-1303}
		p_{(\gamma)}(\x):=\x^r-\sum_{i=0}^{r-1}\varphi_i \x^i\in \RR[\x],
\end{equation}
with $\varphi_i$ as in 
\eqref{def:coef-gen-poly},
the \textbf{generating polynomial} of $\gamma$.

\subsection{Localizing moment matrices}
\label{Hankel-mat-ext}
Let $\gamma$ be as in \eqref{230422-1204}.
The \textbf{Riesz functional} $L:\RR[x]_{\leq D}\to \RR$ of $\gamma$ is defined by $L(x^i):=\gamma_i$ for each $i$.
For $f\in \RR[x]_{\leq D}$ an \textbf{$f$--localizing moment matrix $\cH_{f}$ of $\gamma$} is a real square matrix of size 
$s(D,f)\times s(D,f)$, where 
    $s(D,f)=\lfloor \frac{D-\deg f}{2}\rfloor +1$,
with the $(i,j)$--th entry equal to $L(fx^{i+j-2})$.
We write
\begin{equation}
\label{def:localized}
    f\cdot \gamma:=(\gamma^{(f)}_0,\gamma^{(f)}_1,\ldots,
        \gamma^{(f)}_{D-\deg f}),
        \quad
        \gamma^{(f)}_i:=L(fx^i).
\end{equation}
Note that $\cH_{f}$ corresponds to the moment matrix of the sequence $f\cdot\gamma$. 
We denote the Riesz functional of $f\cdot \gamma$ by $L_{f}$ and call it an \textbf{$f$--localizing Riesz functional of $\gamma$}.

We write $\cH_f(\ell)$ for the $\ell$-th truncated moment matrix of $f\cdot \gamma$, i.e., 
	\begin{equation*}
    \label{localized-truncated}
		\cH_f(\ell):=\left(\gamma^{(f)}_{i+j-2} \right)_{i,j=1}^{\ell+1}
					=	\kbordermatrix{
							& \textit{1} & X & X^2 & \cdots  & X^{\ell} \\[0.2em]
							\textit{1} & \gamma^{(f)}_0 & \gamma^{(f)}_1 & \gamma^{(f)}_2 & \cdots & \gamma^{(f)}_\ell\\[0.2em]
							X & \gamma^{(f)}_1 & \gamma^{(f)}_2 & \iddots & \iddots & \gamma^{(f)}_{\ell+1}\\[0.2em]
							X^2 & \gamma^{(f)}_2 & \iddots & \iddots & \iddots & \vdots\\[0.2em]
							\vdots &\vdots 	& \iddots & \iddots & \iddots & \gamma^{(f)}_{2\ell-1}\\[0.2em]
							X^\ell & \gamma^{(f)}_\ell & \gamma^{(f)}_{\ell+1} & \cdots & \gamma^{(f)}_{2\ell-1} & \gamma^{(f)}_{2\ell}
						}
	\end{equation*}
    
\subsection{Atomic measures}
For $x\in \RR$, $\delta_x$ stands for the Dirac measure supported on $x$.
By a \textbf{finitely atomic positive measure} on $\RR$ we mean a measure of the form 
	$\mu=\sum_{j=1}^\ell \rho_j \delta_{x_j}$, 
where $\ell\in \NN\cup \{0\}$, each $\rho_j>0$ and each $x_j\in \RR$. The points $x_j$ are called 
\textbf{atoms} of the measure $\mu$ and the constants $\rho_j$ the corresponding \textbf{densities}. 

We now recall the definition of the evaluation at $\infty$ from \cite[Definition 1.1]{BKRSV20}. The \textbf{evaluation at $\infty$} is the linear functional, defined by 
\begin{equation}
\label{def:point-infty}
\textbf{ev}_{\infty}:\mathbb R[\x]_{\leq D}\to \RR,
\quad
\sum_{k=0}^{D}f_k\x^k \mapsto f_D.
\end{equation}
Allowing $\textbf{ev}_{\infty}$ to be a part of a finitely atomic positive measure $\mu$, we call $\mu$ a \textbf{generalized measure}. When such a measure represents a sequence $\gamma$ as in \eqref{230422-1204} by \eqref{moment-measure-cond-univ}, then $\mu$ is called a $(\RR\cup \{\infty\})$--rm for $\gamma$.

\subsection{Solution to the $\RR$--TMP}
Let $x_1,\ldots, x_{m}\in \RR$. We denote by $V_{(x_1,\ldots,x_m)}\in \RR^{m\times m}$ the Vandermondo matrix 
	$$V_{(x_1,\ldots,x_m)}:=
		\left(\begin{array}{cccc}
		1 & 1 & \cdots & 1\\
		x_1 & x_2 & \cdots & x_m\\
		\vdots & \vdots &  & \vdots\\
		x_1^{m-1} & x_2^{m-1} & \cdots & x_m^{m-1}
		\end{array}\right).$$ 
The following is a solution to the $\RR$--TMP of degree $2d$.

\begin{theorem}[{\cite[Theorems 3.9 and 3.10]{CF91}}]\label{Hamburger}
	Let $d\in \NN$ and $\gamma=(\gamma_0,\ldots,\gamma_{2d})\in \RR^{2d+1}$ with $\gamma_0>0$. 
	The following statements are equivalent:
\begin{enumerate}[(1)]	
	\item
		\label{pt1-130222-1851} 
		There exists a $\RR$--representing measure for $\gamma$.
	\item 
		There exists a $(\Rank \gamma)$--atomic $\RR$--representing measure for $\gamma$.
	\item\label{pt3-130222-1851} 	
		$\gamma$ is positively recursively generated.
	\item 
		$M_d$ is positive semidefinite and $\Rank M_d=\Rank \gamma$.
	\item\label{pt4-v2206} 
	One of the following statements holds:
	\begin{enumerate}
		\item $M_d$ is positive definite.
		\item $M_d$ is positive semidefinite
			and 
			$\Rank M_d=\Rank M_{d-1}$.
	 \end{enumerate}
\end{enumerate}
Moreover, if a $\RR$--representing measure for $\gamma$ exists, then:
\begin{enumerate}[(i)]
	\item\label{140222-1158}
	If $r\leq d$, then the $\RR$--representing measure is unique and of the form
	$\mu=\sum_{i=1}^{r}\rho_i\delta_{x_i},$ where $x_1,\ldots,x_r$ are the roots of $p_{(\gamma)}$ 
	as in \eqref{230422-1303}
	and
	$
	(\rho_i)_{i=0}^r=	
	V_{(x_1,\ldots,x_m)}^{-1}\mbf{v}_0^{(r-1)}.	
	$
	\item\label{140222-1202} 
	If $r=d+1$, then there are infinitely many $\RR$--representing measures for $\gamma$. All $(d+1)$--atomic ones
		are obtained by choosing $\gamma_{2k+1}\in \RR$ arbitrarily, defining 
		$\gamma_{2k+2}:=(\mbf{v}_{d+1}^{(d)})^T{(M_d})^{-1}\mbf{v}_{d+1}^{(d)}$, 
		and using
		\eqref{140222-1158} for 
		$\widetilde \gamma:=(\gamma_0,\ldots,\gamma_{2d+1},\gamma_{2d+2})\in \RR^{2k+3}.$ 
\end{enumerate}
\end{theorem}

\subsection{Companion matrix}	
    For a univariate polynomial
		$p(\x)=\x^k-\sum_{i=0}^{k-1}\lambda_{i}\x^{i}\in \RR[\x]$
	let
	\begin{equation}
	\label{230424-1153}
		C_{p(\x)}=
		\begin{pmatrix}
		0 & \cdots & && 0 & \lambda_0\\
		1 & 0 &&& 0& \lambda_1\\
		0 & 1 & \ddots && \vdots & \vdots \\
		\vdots &\ddots & \ddots& \ddots & \vdots&\vdots \\
		&&\ddots & 1 &0 & \lambda_{k-2}\\
		0& \cdots & \cdots& 0 & 1 & \lambda_{k-1}
		\end{pmatrix} \in \RR^{k\times k}
	\end{equation}
	be its \textbf{companion matrix}, i.e.,
		$p(\x)=\det(\x I_{k}-C_{p(\x)}),$
	where $I_k$ stands for the identity matrix of size $k$.

\subsection{Elementary symmetric polynomials}

Let $d_1\in \NN$ and 
  $$
  e_i(\mathtt{x}_1,\ldots,\mathtt{x}_{d_1}):=
  \sum_{1\leq j_1<j_2<\ldots<j_i\leq d_1}
  \mathtt{x}_{j_1}\mathtt{x}_{j_1}\cdots\mathtt{x}_{j_i}
  $$
  stand for the $i$-th elementary symmetric polynomial in variables 
  $\mathtt{x}_1,\ldots,\mathtt{x}_{d_1}$.
  Given distinct real numbers 
  $x_1,\ldots,x_{d_1}$, 
  let $e_i$ stand for $e_i(x_1,\ldots,x_{d_1})$. In particular,
\begin{equation}
\label{eq:sym-polys}
e_0=1,\quad 
e_1 = \sum_{i=1}^{d_1} x_i, \quad 
e_2 = \sum_{1 \leq i < j \leq d_1} x_i x_j, \quad \dots, \quad e_{d_1} = x_1 x_2 \cdots x_{d_1}.
\end{equation}

\section{Minimal measures with finitely many prescribed atoms}

The main results of this section are solutions to Problem \ref{problem} from Introduction, i.e., Theorem \ref{thm:main} covers the case where only real atoms are allowed in the representing measure, while in Theorem \ref{the:main-gene}
the functional $\textbf{ev}_{\infty}$ (see \eqref{def:point-infty})
is also allowed in the support of the measure.

\begin{theorem}
    \label{thm:main}
	Let $d_1,d_2\in \NN$ 
        and 
        $x_1,\ldots,x_{d_1}\in \RR$ be distinct real numbers. Let $D:=d_1+2d_2-1$. Assume that $\gamma\equiv \gamma^{(D)}=(\gamma_0,\ldots,\gamma_{D})\in\RR^{D+1}$ is a sequence such that 
    the moment matrix	
        $M_{\lfloor\frac{D}{2}\rfloor}$ is positive definite.
    Let $e_i:=e_i^{(d_1)}$ be as in \eqref{eq:sym-polys}
    and
    $$
    f(\x)
    :=\prod_{i=1}^{d_1}(\x - x_i)
    =\sum_{i=0}^{d_1} (-1)^ie_i \x^{d_1-i}.
    $$
	The following statements are equivalent:
	\begin{enumerate}
	\item\label{equi-1}
		There exists a $(d_1+d_2)$--atomic $\RR$--representing measure for $ \gamma$ with $d_1$ atoms equal to $x_1,\ldots,x_{d_1}.$
  
         \item\label{equi-2} 
         The following conditions hold:
         \begin{enumerate} 
            \item\label{equi-2-cond-1}
                The localizing matrix $\cH_f(d_2-1)$
                is invertible. 
            \item\label{equi-2-cond-2} 
                Denote
               \begin{equation}
               \label{def:lambda}
    \begin{pmatrix}
    \lambda_0 & \lambda_1 & \cdots & \lambda_{d_2-1}
    \end{pmatrix}^T
    =
    \big(\cH_f(d_2-1)\big)^{-1}
    \Big(\sum_{i=0}^{d_1}(-1)^i e_i\mathbf{v}^{(d_2-1)}_{d_1+d_2-i}\Big)
                \end{equation}
                where 
                $\mathbf{v}_{d_1+d_2-i}^{(d_2-1)}$ are as in \eqref{bold-v},
                let
                \begin{align}
                \label{def:q-p-theo}
                \begin{split}
                h(\x)
                &:=
                f(\x)\underbrace{\Big(\x^{d_2}-\sum_{i=0}^{d_2-1}\lambda_i \x^i\Big)}_{g(\x)}
                =
                \x^{d_1+d_2}
                -
                \sum_{i=0}^{d_1+d_2-1}
                \varphi_i \x^i,
                \end{split}
                \end{align}
                and let  
                $$\widetilde{\gamma}:=\{\widetilde{\gamma}_u\}_{u=0}^{D+d_1-1}$$
                be the extension of $\gamma$, defined by
                $\widetilde \gamma_u=\gamma_u$ for $0\leq u\leq D$
                and
                \begin{equation}
                \label{eq:new-moments-v2}
                \widetilde \gamma_{D+\ell}
      =
                \sum_{i=1}^{d_1+d_2}
                \varphi_{d_1+d_2-i} \widetilde \gamma_{D+\ell-i}
                \quad\text{for}\quad
                \ell=1,\ldots,d_1-1.
                \end{equation}
                The moment matrix $M_{d_1+d_2-1}$ of $\widetilde{\gamma}$ is positive definite.
                \end{enumerate}
	\end{enumerate}
    
    Moreover, if the equivalent statements \eqref{equi-1},\eqref{equi-2} hold, then
    the other atoms in the measure are the zeros of the polynomial 
    $g(\x)$.
\end{theorem}

\begin{theorem}
\label{the:main-gene}
Let $d_1,d_2\in \NN$
and 
$x_1,\ldots,x_{d_1}\in \RR$ be distinct real numbers. 
Let $D:=d_1+2d_2-1$. Assume that $\gamma\equiv \gamma^{(D)}=(\gamma_0,\ldots,\gamma_{D})\in\RR^{D+1}$ is a sequence such that 
    the moment matrix	
        $M_{\lfloor\frac{D}{2}\rfloor}$ is positive definite.
    The following statements are equivalent:
 \begin{enumerate}
     \item
     \label{the:main-gene-pt1}
     There exists a $(d_1+d_2)$--atomic $(\RR\cup \{\infty\})$--representing measure $\mu$ for $\gamma$ with $d_1$ atoms equal to $x_1,\ldots,x_{d_1}$.
     \item  
     \label{the:main-gene-pt2}
     One of the following statements hold:
     \begin{enumerate}
     \item
     \label{the:main-gene-pt2a}     
     There exists a $(d_1+d_2)$--atomic $\RR$--representing measure $\mu$
     for $\gamma$
     with $d_1$ atoms equal to $x_1,\ldots,x_{d_1}$, obtained by Theorem \ref{thm:main}.
     \item
     \label{the:main-gene-pt2b}
     There exists a $(d_1+d_2-1)$--atomic $\RR$--representing measure $\mu$ 
     for the sequence
     $\widehat{\gamma}\equiv (\gamma_0,\ldots,\gamma_{D-2})\in\RR^{D-1}$
     with $d_1$ atoms equal to $x_1,\ldots,x_{d_1}$,
     which also represents $\gamma_{D-1}$ and $\gamma_D-\alpha$ for some $\alpha > 0$.
     Namely, 
     denoting $e_i:=e_i^{(d_1)}$ as in \eqref{eq:sym-polys}
    and
    $
    f(\x)
    :=\prod_{i=1}^{d_1}(\x - x_i),
    $
    the following conditions hold:
    \smallskip
     \begin{enumerate} 
            \item\label{point2-1}
                The localizing matrix $\cH_f(d_2-2)$
                is invertible. 
            \item \label{point2-2}
                Denote
               \begin{equation*}
    \begin{pmatrix}
    \lambda_0 & \lambda_1 & \cdots & \lambda_{d_2-2}
    \end{pmatrix}^T
    =
    \big(\cH_f(d_2-2)\big)^{-1}
    \Big(\sum_{i=0}^{d_1}(-1)^i e_i\mathbf{v}^{(d_2-2)}_{d_1+d_2-1-i}\Big)
                \end{equation*}
                where 
                $\mathbf{v}_{d_1+d_2-1-i}^{(d_2-2)}$ are as in \eqref{bold-v},
                let
                \begin{align*}
                \begin{split}
                h(\x)
                &:=
                f(\x)\Big(\x^{d_2-1}-\sum_{i=0}^{d_2-2}\lambda_i \x^i\Big)
                =
                \x^{d_1+d_2-1}
                -
                \sum_{i=0}^{d_1+d_2-2}
                \varphi_i \x^i,
                \end{split}
                \end{align*}
                and let  
                $$\widetilde{\gamma}:=\{\widetilde{\gamma}_u\}_{u=0}^{D+d_1-3}$$
                be defined by
                $\widetilde \gamma_u=\gamma_u$ for $0\leq u\leq D-2$
                and
                \begin{equation*}
                \widetilde \gamma_{D-2+\ell}
      =
                \sum_{i=1}^{d_1+d_2-1}
                \varphi_{d_1+d_2-1-i} \widetilde \gamma_{D-2+\ell-i}
                \quad\text{for}\quad
                \ell=1,\ldots,d_1-1.
                \end{equation*}
                The moment matrix $M_{d_1+d_2-2}$ of $\widetilde{\gamma}$ is positive definite,
                $$\widetilde \gamma_{D-1}=\gamma_{D-1} 
                \quad\text{and}\quad 
                \widetilde\gamma_{D}<\gamma_D.$$
                \end{enumerate}
     \end{enumerate}
 \end{enumerate}
\end{theorem}

\begin{remark}
\begin{enumerate}
    \item\label{comment-point1} 
        In the case $d_1=1$, there is nothing to check in \eqref{equi-2-cond-2}
        of Theorem \ref{thm:main} and the invertibility of $\cH_f(d_2-1)$ is equivalent to the existence of a $\RR$--rm. This agrees with Theorem \cite[Theorem 1.4]{BKRSV20}.
        The remaining nodes in \cite[Theorem 1(a)]{BKRSV20} are solutions of $F(x_1,\y)=0$ where
        $F(\Tt,\y):=
        \det{\big(\cH_{(\x-\Tt)(\x-\y)}(d_2-1)\big)}.
        $
        Note that 
        $$\cH_{(\x-\y)f(\x)}(d_2-1)
        =
        \cH_{\x f(\x)}(d_2-1)-\y\cH_{f(\x)}(d_2-1)
        =
        \cH_{f(\x)}(d_2-1)(C_{g(\x)}-\y I),
        $$
        where $C_{g(\x)}$ is the companion matrix of $g(\x)$ from \eqref{def:q-p-theo}. So $F(x_1,y)=0$ indeed corresponds to $g(y)=0$, which should be true 
        by the moreover part of Theorem \ref{thm:main}.
        \smallskip
    \item 
        Assume that $d_1=1$ and $\cH_{f}(d_2-1)$ is singular. 
        Let 
        $\begin{pmatrix}
            \varphi_0 & \cdots & \varphi_{d_{2}-1}
        \end{pmatrix}^T=
        M_{d_2-1}^{-1}\mathbf{v}_{d_2}^{(d_2-1)}
        $
        and $p(\x)=\x^{d_2}-\sum_{i=0}^{d_2-1}\varphi_i \x^i$.
        Since
        $\cH_{\x}(d_2-1)=M_{d_{2}-1}C_{p(\x)}$,
        where $C_{p(\x)}$ is the companion matrix of $p(\x)$, it follows that
        $\cH_{f(\x)}(d_2-1)=M_{d_2-1}(C_{p(\x)}-x_1I)$.
        Hence, $x_1$ is a zero of $p(\x)$ and by Theorem \ref{Hamburger} zeroes
        of $p(\x)$ represent the support of the $d_2$--atomic $\RR$--rm for 
            $(\gamma_0,\ldots,\gamma_{D-1},\gamma_D-M_{d_2}/M_{d_2-1})$,
        where $M_{d_2}/M_{d_2-1}$ stands for the Schur complement of $M_{d_2-1}$ in $M_{d_2}$. 
        Since $\alpha:=\gamma_D-M_{d_2}/M_{d_2-1}$ is positive, the remaining atom is
        $\mathbf{ev}_{\infty}$ with the density $\alpha$.
    \item Note that the representing measures obtained by Theorems \ref{thm:main} and       \ref{the:main-gene} are unique.
    \item Theorem \ref{thm:main} and \ref{the:main-gene} assume $d_1+d_2$ atoms with
        $x_1,\ldots,x_{d_1}$ prescribed are necessary to represent $\gamma$. It might happen that there already exists a $(d_1+i)$--atomic representing measure for some 
        $1\leq i< d_2$. One can also handle these cases by the following simple adaptation. Namely, for a given $i$ Theorem \ref{thm:main} is applied to the truncation $\gamma^{(D_i)}=\{\gamma_{j}\}_{j=0}^{D_i}$ where $D_i=d_1+2i-1$ of $\gamma$.
        Since the measure obtained by Theorem \ref{thm:main} is unique, one only checks whether this measure also represents $\gamma_{D_i+1},\ldots,\gamma_{D-1}$ and $\gamma_{D}-\alpha$ for some $\alpha\geq 0$. If $\alpha>0$, then 
        $\mathbf{ev}_\infty$ must also be added as an atom with the density $\alpha$.
\end{enumerate}
\end{remark}

The following two lemmas are important technical tools in the proof of Theorem \ref{thm:main}.

\begin{lemma}\label{prod-A-i}
    Let $k,d \in \mathbb N$ with $k < d$.
    Given distinct real numbers
    $x_1,\ldots,x_k$,
    write
    \begin{equation*} 
		A_i:=
		\begin{pmatrix}
		-x_i & 1 & 0  & \cdots & 0\\
		0 & -x_i & 1&\ddots & \vdots \\
		\vdots  & \ddots & \ddots &  \ddots & 0 \\
		0 & \cdots & 0 &  -x_i & 1
		\end{pmatrix}\in \RR^{(d-i+1)\times (d-i+2)}
            \quad
            \text{for}\quad i=1,\ldots,k.
   \end{equation*}
   and
    \begin{equation} 
    \label{def:Bk}
    B_k:=A_k A_{k-1} \cdots A_1\in \RR^{(d-k+1)\times (d+1)}.
    \end{equation}
   Let $e_i:=e_i^{(k)}$ be as in \eqref{eq:sym-polys} (with $d_1=k$).
   Then
   \begin{tiny}
    \begin{equation}
    \label{prod-mat}
	    B_k=
		\begin{pmatrix}
		(-1)^ke_k & (-1)^{k-1}e_{k-1} &  \cdots  & (-1)^{k-i}e_{k-i} & \cdots & e_{0} & 0 & \cdots & \cdots & 0 \\[0.2em]
		0 & (-1)^ke_k & (-1)^{k-1}e_{k-1} & \cdots &(-1)^{k-i}e_{k-i} & \cdots & e_0 & 0 &  & 0 \\[0.2em]
		\vdots  & \ddots & \ddots &  \ddots & &\ddots &&\ddots &\ddots &\vdots \\
        \vdots  &  & \ddots &  \ddots &\ddots &&\ddots&&\ddots&0 \\[0.2em]
		0 & \cdots & \cdots & 0 &  (-1)^ke_k & (-1)^{k-1}e_{k-1} &\cdots & (-1)^{k-i}e_{k-i}& \cdots & e_0
		\end{pmatrix}. 
    \end{equation}
    \end{tiny}
    Equivalently, writing 
    $B_k=(b_{ij})_{i,j}$, we have that
    \begin{equation}
    \label{eq-entry-ij}
    b_{ij}=
    \left\{
    \begin{array}{rl}
    (-1)^{k+i-j}e_{k+i-j}, &  1\leq i\leq j \leq i+k,\\[0.2em]
    0, & \text{otherwise}.
    \end{array}
    \right.
    \end{equation}
\end{lemma}
\begin{proof}
    We prove the result inductively on $k.$ 
    The basis case $k=1$ is clear.
    Now, assume that \eqref{prod-mat} holds for $k-1<d-1$ and prove it for $k<d$.
    Namely, assuming that
  $$
  B_{k-1}=(b_{ij}^{(k-1)})_{i,j}\in \RR^{(d-k+2)\times (d+1)}
  $$
  with
    \begin{equation*}
    b_{ij}^{(k-1)}=
    \left\{
    \begin{array}{rl}
    (-1)^{k-1+i-j}e^{(k-1)}_{k-1+i-j}, &  1\leq i\leq j \leq i+k-1,\\[0.2em]
    0, & \text{otherwise},
    \end{array}
    \right.
    \end{equation*}
 we have to prove that
 $B_k=A_kB_{k-1}$ is of the form \eqref{prod-mat} or equivalently \eqref{eq-entry-ij} holds.
 We have that
 \begin{align*}
    \begin{split}
    b_{ij}
    &=
    \begin{pmatrix}
        \smash[b]{\block{i-1}} & -x_k & 1 & \smash[b]{\block{d-k-i}} 
    \end{pmatrix}
    \begin{pmatrix}
        (-1)^{k-j}e_{k-j}^{(k-1)} \\[0.2em]
        (-1)^{k-j+1}e_{k-j+1}^{(k-1)}  \\[0.2em]
        \vdots \\[0.2em]
        (-1)^{k-j+d-k}e_{d-j-1}^{(k-1)}
    \end{pmatrix}\\
    &=(-1)^{k+i-j-1}(-x_k)e_{k-j+i-1}^{(k-1)}+(-1)^{k+i-j}e_{k-j+i}^{(k-1)}\\[0.2em]
    &=(-1)^{k+i-j}(x_ke_{k-j+i-1}^{(k-1)}+e_{k-j+i}^{(k-1)})\\[0.2em]
    &=(-1)^{k+i-j}e_{k+i-j}^{(k)}.
    \end{split}
 \end{align*}
 This completes the proof.
\end{proof}
\allowdisplaybreaks

\begin{lemma}
    \label{lem:aux-2}
    Let $k,d \in \mathbb N$ with $k < d$.
    Let $\gamma\equiv (\gamma_0,\gamma_1,\ldots,\gamma_{2d})\in \RR^{2d+1}$ with a corresponding moment matrix $M_d$.
    Given distinct real numbers
    $x_1,\ldots,x_k$, let $e_i:=e_i^{(k)}$ be as in \eqref{eq:sym-polys} (with $d_1=k$)
    and
    $
    f(\x)
    :=\prod_{i=1}^{k}(\x - x_i)=\sum_{i=0}^{k}(-1)^i e_i\x^{k-i}.
    $
    Let $B_k$ be as in \eqref{def:Bk}
    and $\gamma^{(f)}_i$ as in \eqref{def:localized}.
    Then
        $$(B_kM_{d})_{ij}=\gamma^{(f)}_{i+j-2}
        \quad 
        \text{for}
        \quad 
        1\leq i\leq d-k+1,\;
        1\leq j\leq d+1,
        $$
    and
        $$
        B_kM_dB_k^T
        =
        \sum_{i=0}^{k}
        (-1)^{i}e_i\cH_{\x^{k-i}f(\x)}(d-k).
        $$
        
    Moreover, if $M_d$ is positive definite, 
    then $B_kM_dB_k^T$ is invertible.
\end{lemma}

\begin{proof}
By definition of $\gamma^{(f)}_{i}$, we have to prove that
\begin{equation}
\label{eq:localizing}
    (B_kM_d)_{ij}=L(f\x^{i+j-2})
    \quad
    \text{for each}\quad 1\leq i\leq d-k+1,\; 1\leq j\leq d+1,
\end{equation}
where $L$ is a Riesz functional of $\gamma$.
Let $\mbf{v}_i^{(j)}:=\left( \gamma_{i+r-1} \right)_{1\leq r\leq j+1}$
and let $L$ be the Riesz functional of $\gamma$.
We have that:
\begin{align*}
    (B_kM_d)_{ij}
    &=
        \big(B_k
        \begin{pmatrix} 
            \mathbf{v}_{0}^{(d)} &
            \cdots
            \mathbf{v}_{j-1}^{(d)} &
            \cdots &
            \mathbf{v}_{d}^{(d)} &
        \end{pmatrix}\big)_{ij}\\
    &=
        \sum_{\ell=1}^{d+1} b_{i\ell}\gamma_{j+\ell-2}\\
    &\underbrace{=}_{\eqref{eq-entry-ij}}
        \sum_{\ell=i}^{i+k} (-1)^{k+i-\ell} e_{k+i-\ell}\gamma_{j+\ell-2}\\
    &=
        L\Big(\sum_{\ell=i}^{i+k} (-1)^{k+i-\ell} e_{k+i-\ell}\x^{j+\ell-2}\Big)\\
    &=
        L\Big(\Big(\sum_{\ell=i}^{i+k} (-1)^{k+i-\ell} e_{k+i-\ell}\x^{\ell-i}\Big)
        \x^{i+j-2}\Big)\\
    &=^6 L\Big(\Big(\sum_{\widehat\ell=0}^{k} (-1)^{k-\widehat\ell} e_{k-\widehat\ell}\;\x^{\widehat\ell}\Big)
        \x^{i+j-2}\Big)\\
    &=^7 L\Big(\Big(\sum_{\widetilde\ell=0}^{k} (-1)^{\widetilde\ell} e_{\widetilde\ell}\; \x^{k-\widetilde\ell}\Big)
        \x^{i+j-2}\Big)\\
    &=
        L(f\x^{i+j-2}),
\end{align*}
where in the sixth (resp.\ the seventh) equality we introduced a new variable $\widehat\ell=\ell-i$ (resp.\ $\widetilde \ell= k-\widehat\ell$).

Further on,
\begin{align*}
    B_kM_dB_k^T
    &=(\gamma^{(f)}_{i+j-2})_{  
        \substack{
            1\leq i\leq d-k+1,\\
            1\leq j\leq d+1
            }} B_k^T\\
    &=\sum_{\ell=0}^{k}
        (-1)^{k-\ell} e_{k-\ell} 
        (\gamma^{(f)}_{i+j-2+\ell})_{  
            i,j=1
            }^{d-k+1}\\
    &=\sum_{\ell=0}^{k}
        (-1)^{k-\ell} e_{k-\ell}
            \cH_{\x^{\ell}f(\x)}(d-k)\\
     &=\sum_{i=0}^{k}
        (-1)^{i} e_{i}
            \cH_{\x^{k-i}f(\x)}(d-k),
\end{align*}
where we introduced a new variable $i=k-\ell$ in the last equality.

    The moreover part follows by noticing that $\Rank B_k=d-k+1$.
    Indeed, if all $x_i$ are nonzero, then $e_k\neq 0$. So the first diagonal of $B_k$ is nonzero and $\Rank B_k=d-k+1$. 
    If one of $x_i$ is 0, then we may assume that $x_1=0$ and hence 
    $e_{k-1}=x_2\cdots x_k\neq 0$. Thus the second diagonal of $B_k$ is nonzero
    and $\Rank B_k=d-k+1$.
    Now for $v\in \RR^{d-k+1}$ we have	
	\begin{align*}
		B_kM_{d}B_k^Tv=0
		\quad&\Rightarrow\quad
		v^TB_kM_{d}B_k^Tv=
		(B_k^Tv)^TM_{d}(B_k^Tv)=0\\
		&\Rightarrow\quad
		B_k^Tv=0\\
		&\Rightarrow\quad
		v=0,
	\end{align*}
	where we used that $M_{d}$ is positive definite in the second and the fact that $B_k^T$ has a trivial kernel (being of full column rank) in the last line.
    Therefore $B_kM_{d}B_k^T$ is invertible.
\end{proof}

Now we are ready to prove Theorem \ref{thm:main}.
    \begin{proof}[Proof of Theorem \ref{thm:main}]
    First, we prove the implication $\eqref{equi-1}\Rightarrow	\eqref{equi-2}$.
	Let 
		$\mu=\sum_{i=1}^{d_1+d_2}\rho_i\delta_{y_i}$
	be a $(d_1+d_2)$--atomic $\RR$--rm for $\gamma$ with $\rho_i>0$ and $y_i=x_i$ for $i= 1, \ldots, d_1$.
	We extend $M_{\lfloor\frac{D}{2}\rfloor}$ to $M_{d_1+d_2}$ by computing $\gamma_{D+1},\gamma_{D+2},\ldots, \gamma_{D+d_1+1}$ with respect to the measure  $\mu$:
    \begin{equation*}
	M_{d_1+d_2}=
	\kbordermatrix{
		& \textit{1} & X & \cdots & X^{\lfloor\frac{D}{2}\rfloor}  &\cdots & X^{d_1+d_2} \\
	(\vec{X})^T &
		\mathbf{v}_0^{(d_2-1)} & \mathbf{v}_1^{(d_2-1)} & \cdots & \mathbf{v}_{\lfloor\frac{D}{2}\rfloor}^{(d_2-1)} & \cdots & \mathbf{v}_{d_1+d_2}^{(d_2-1)}\\
	X^{d_2} &
		\gamma_{d_2} & \gamma_{d_2+1} & \cdots & \gamma_{\lfloor\frac{D}{2}\rfloor+d_2}&\cdots & \gamma_{D+1}  \\
  \vdots & \vdots &\vdots  &  &\vdots &\ddots & \vdots \\
  X^{d_1+d_2} &
		\gamma_{d_1+d_2} & \gamma_{d_1+d_2+1} & \cdots & \gamma_{\lfloor\frac{D}{2}\rfloor+d_1+d_2}  &\cdots & \gamma_{D+d_1+1}
	},
	\end{equation*}
	where 
	$\vec{X}
	=
	\begin{pmatrix}
		\textit{1} & X & \cdots & X^{d_2-1}
	\end{pmatrix}.
	$
 Let
	\begin{align}
    \label{def-q-p}
    \begin{split}
		\widetilde g(\x)
		&:=\prod_{i=d_1+1}^{d_1+d_2}(\x-y_i)
		=\x^{d_2}-\sum_{i=0}^{d_2-1}\widetilde\lambda_i \x^i,\\
		\widetilde h(\x)
		&:= \prod_{i=1}^{d_1+d_2}(\x-y_i)=
                \Big(\prod_{i=1}^{d_1}(\x - x_i)\Big)\widetilde g(\x)\\
            &=\Big(\sum_{i=0}^{d_1} (-1)^i e_i \x^{d_1-i}\Big)
                    \Big(\x^{d_2}-\sum_{i=0}^{d_2-1}\widetilde\lambda_i \x^i\Big)\\
            &=\Big(\sum_{i=0}^{d_1} (-1)^i e_i \x^{d_1-i}\Big)
                    \x^{d_2}-
              \Big(\sum_{i=0}^{d_1} (-1)^i e_i \x^{d_1-i}\Big)\Big(\sum_{i=0}^{d_2-1}\widetilde\lambda_i \x^i\Big).
    \end{split}
	\end{align}
We will prove that $\widetilde g(\x)=g(\x)$ and $\widetilde h(\x)=h(\x)$, 
    where $g(\x)$ and $h(\x)$ are as in \eqref{def:q-p-theo}.
    By Theorem \ref{Hamburger}, 
$M_{d_1+d_2-1}$ is invertible and $\widetilde h(X)=\mbf{0}$ in $M_{d_1+d_2}$.
     By Lemma \ref{lem:aux-2} used for $k=d_1$ and $d:=d_1+d_2-1$,
    we get that
    \begin{equation}
        \label{invertibility} 
        \sum_{i=0}^{d_1}
        (-1)^{i}e_i\cH_{\x^{d_1-i}f(\x)}(d_2-1)
        \quad\text{is invertible}.
    \end{equation}
Let $C_{\widetilde g(\x)}$ is the companion matrix of $\widetilde g(\x)$ (see \eqref{230424-1153}).
    From the third equality for $\widetilde h(\x)$ in 
    \eqref{def-q-p}
    above it follows that  
    \begin{equation}
        \label{rel:companion}
            \cH_{x^{d_1-i}f(\x)}(d_2-1)
            =\cH_{x^{d_1-i-1}f(\x)}(d_2-1)C_{\widetilde g(\x)}
            \quad\text{for}\quad
            i=0,\ldots,d_1-1.
    \end{equation}
    Using \eqref{rel:companion} in \eqref{invertibility}, it follows that
    \begin{equation*} 
        \cH_{f}(d_2-1)
        \left(
        \sum_{i=0}^{d_1}
        (-1)^{i}e_i C_{\widetilde g(\x)}^{d_1-i}
        \right)
        \quad\text{is invertible}.
    \end{equation*}
    In particular, $\cH_f(d_2-1)$ is invertible. \eqref{rel:companion} used for 
    $i=d_1-1$, implies that $\widetilde\lambda_i=\lambda_i$ for each $i$, where $\lambda_i$ are as in \eqref{def:lambda}. This proves that $g(\x)=\widetilde g(\x)$ and  
    $h(\x)=\widetilde h(\x)$. Finally, observing the $d_1-1$ entries in rows 
    $X^{d_2},\ldots,X^{d_1+d_2-2}$ of the column $X^{d_1+d_2}$ of $M_{d_1+d_2}$ gives equalities \eqref{eq:new-moments-v2}.
     This
concludes the proof of the implication $\eqref{equi-1}\Rightarrow	\eqref{equi-2}$.\\

It remains to prove the implication \eqref{equi-2}$\Rightarrow$ \eqref{equi-1}.
We define $\gamma_{D+d_1}$, $\gamma_{D+d_1+1}$ by \eqref{eq:new-moments-v2} for $\ell=d_1,d_1+1$.
By assumptions of the theorem, the moment matrix $M_{d_1+d_2}$
satisfies the column relation 
    $$X^{d_1+d_2}-\sum_{i=0}^{d_1+d_2-1}\varphi_i X^i=\mbf{0}.$$
Since $M_{d_1+d_2-1}$ is positive definite,  Theorem \ref{Hamburger} implies that there is a $(d_1+d_2)$--atomic $\RR$--rm for $\gamma$ supported on the zeroes of this column relation. These are $x_1,\ldots,x_{d_1}$ and the zeroes of $g(\x)$.
\end{proof}

Finally, we prove Theorem \ref{the:main-gene}.

\begin{proof}[Proof of Theorem \ref{the:main-gene}]
The nontrivial implication is $\eqref{the:main-gene-pt1}\Rightarrow	\eqref{the:main-gene-pt2}$. We only need to prove that if \eqref{the:main-gene-pt2a} does not hold, then \eqref{the:main-gene-pt2b} holds. Assume that
\eqref{the:main-gene-pt2a} does not hold. Then one of the atoms in a $(d_1+d_2)$--atomic $(\RR\cup \{\infty\})$--rm $\mu$ for $\gamma$ must be $\infty$.
Hence, $\mu=\sum_{j=1}^{d_1+d_2-1}\rho_j \delta_{y_j}+\rho_{d_1+d_2}\textbf{ev}_{\infty}$
for some $\rho_j>0$ and $y_j\in \RR$ with $y_1=x_1,\ldots,y_{d_1}=x_{d_1}$. Thus, we have
    \begin{equation*}
        \displaystyle
        \gamma_i
        =\sum_{j=1}^{d_1+d_2-1}\rho_jy^i_j +\rho_{d_1+d_2}\textbf{ev}_{\infty}(x^i)
        =
        \displaystyle
        \left\{
        \begin{array}{rl}
        \displaystyle\sum_{j=1}^{d_1+d_2-1}\rho_jy^i_j,& \text{if }i<D,\\[0.3em]
        \displaystyle\sum_{j=1}^{d_1+d_2-1}\rho_jy^i_j +\rho_{d_1+d_2},& \text{if }i=D.
        \end{array}
        \right.
    \end{equation*}
Write
$$
\gamma=\underbrace{(\gamma_0,\gamma_1,\ldots,\gamma_{D-2},0,0)}_{(\widehat{\gamma},0,0)}+(\underbrace{0,\ldots,0}_{D-1},\gamma_{D-1},\rho_{d+1}).
$$
Since $\widehat{\gamma}$ has a unique $(d_1+d_2-1)$--atomic $\RR$--rm $\sum_{j=1}^{d_1+d_2-1}\rho_j \delta_{y_j}$ such that $y_j=x_j$ for $j=1,\ldots,d_1$, 
the statement \eqref{the:main-gene-pt2b} follows by Theorem \ref{thm:main}.
\end{proof}
The next numerical example (\url{https://github.com/ZalarA/Quadratures}, obtained using \cite{Wol}) shows the violation of the conditions in Theorems \ref{thm:main} and \ref{the:main-gene}, which prevents the existence of the measure for a given sequence of degree 9 with two prescribed atoms. This is
\cite[Example 2]{BKRSV20}, where the authors show that the necessary condition $\det(\cH_{(\x-y)f(\x)}(d_2-1))=0$ for the existence of a measure containing  $x_1,\ldots,x_{d_1}$ and $y$ in the support is not also sufficient (see \cite[Proposition 4.1]{BKRSV20}).

\begin{example}
    \label{ex:no-measure}
    Let $\gamma\equiv \{\gamma_i\}_{i=0}^9\in \RR^{10}$
    be a sequence defined by $\gamma_i=i!$. Namely, 
        $$\gamma=(1,1,2,6,24,120,720,5040,40320,362880).$$
    Let $x_1=\frac{1}{3}$, $x_2=11$ and $d_1=2, d_2=4$.
    The question is whether there is a $(d_1+d_2)$--atomic $(\RR\cup \{\infty\})$--rm for 
    $\gamma$ with two atoms equal to $x_1$, $x_2$. 
    
    First we will demonstrate which condition in Theorem \ref{thm:main} is violated,
    preventing the existence of a
    $6$--atomic $\RR$--rm for $\gamma$ having atoms $x_1,x_2$ in the support.
    Using the notation of Theorem \ref{thm:main}, let 
    $f(\x)=(\x-\frac{1}{3})(\x-11)$ and
    $$
    \cH_f(3)=
    \begin{pmatrix}
        -\frac{17}{3} & -13 & -\frac{110}{3} & -130 \\[0.2em]
        -13 & -\frac{110}{3} & -130 & -552 \\[0.2em]
        -\frac{110}{3} & -130 & -552 & -2680 \\[0.2em]
        -130 & -552 & -2680 & -14160
     \end{pmatrix}.
    $$
    The eigenvalues of $\cH_f(3)$ are 
        $-14691.9, -61.04, -1.54, 0.18$ 
    and hence $\cH_f(3)$ is invertible. The last column of $\cH_f(4)$, restricted to the first 4 rows, is
    $$v=
    \begin{pmatrix}
    -552 & -2680 & -14160 & -75600    
    \end{pmatrix}^T.
    $$
    Hence,
    $
    \big(H_f(3)\big)^{-1}v=
    \begin{pmatrix}
    -\frac{46998216}{137503},\frac{41197920}{137503},-\frac{11282760}{137503},\frac{1695024}{137503}
    \end{pmatrix}^T.
    $
    So 
        $$g(\x)=\x^4-\frac{1695024 }{137503}\x^3+\frac{11282760 }{137503}x^2-\frac{41197920 }{137503}x+\frac{46998216}{137503}$$
    and $h(\x)=f(\x)g(\x)=\x^6-\sum_{i=0}^5\varphi_i \x^i$ is equal to
    $$
    h(\x)
    =
    \x^6-
    \frac{9760174 }{412509}\x^5+
    \frac{92991629 }{412509}\x^4-
    \frac{175284288 }{137503}x^3
    +\frac{555278096 }{137503}\x^2-
    \frac{683705488}{137503}\x+
    \frac{172326792}{137503}.
    $$
    The candidate for $\gamma_{10}$, coming from the potential measure for $\gamma$,
    is
    $$
    \gamma_{10}
    =
    \sum_{i=1}^{6}\varphi_{6-i} \gamma_{10-i}
    =\frac{492324551232}{137503}.
    $$
    The extended moment matrix $M_5$ is then equal to
    $$
    M_5
    =
    \begin{pmatrix}
        1 & 1 & 2 & 6 & 24 & 120 \\[0.2em]
        1 & 2 & 6 & 24 & 120 & 720 \\[0.2em]
        2 & 6 & 24 & 120 & 720 & 5040 \\[0.2em]
        6 & 24 & 120 & 720 & 5040 & 40320 \\[0.2em]
        24 & 120 & 720 & 5040 & 40320 & 362880 \\[0.2em]
        120 & 720 & 5040 & 40320 & 362880 & \frac{492324551232}{137503}
    \end{pmatrix}.
    $$
    Its eigenvalues are 
    $3.62\cdot 10^6, 3774.42, 14.69, 0.808, 0.0398, \mathbf{-0.436}$.
    Since the smallest one is negative, $M_5$ is not positive definite, which is needed for the existence of the measure.

    Finally we will show that also if we allow the atom at $\infty$ there is still no 
    $6$-atomic $(\RR\cup \{\infty\})$--rm for $\gamma$ with atoms $x_1,x_2$ in the support.
    It remains to consider the case when $\infty$
    is necessarily one of the atoms. So we have to find a candidate for a 5--atomic $\RR$--rm with $x_1,x_2$
    as atoms for the truncated sequence $\widehat \gamma=\{\gamma_i\}_{i=0}^7$ and then check whether the measure obtained also represents $\gamma_8$ and if its moment of degree 9 is strictly smaller than $\gamma_9$. 
    We have that
    $$
    \cH_f(2)=
    \begin{pmatrix}
        -\frac{17}{3} & -13 & -\frac{110}{3}  \\[0.2em]
        -13 & -\frac{110}{3} & -130  \\[0.2em]
        -\frac{110}{3} & -130 & -552 
     \end{pmatrix}.
    $$
    The eigenvalues of $\cH_f(2)$ are 
        $-585.52, -8.76, -0.054$ 
    and hence $\cH_f(2)$ is invertible. The last column of $\cH_f(3)$, restricted to the first 3 rows, is
    $$v_2=
    \begin{pmatrix}
    -130 & -552 & -2680    
    \end{pmatrix}^T.
    $$
    Hence,
    $
    \big(H_f(2)\big)^{-1}v_2=
    \begin{pmatrix}
    \frac{204774}{1861} & -\frac{172062}{1861} & \frac{35955}{1861}
    \end{pmatrix}^T.
    $
    So 
        $$g_2(\x)=\x^3-\frac{35955}{1861}\x^2+\frac{172062}{1861}\x-\frac{204774}{1861}$$
    and $h_2(\x)=f(\x)g_2(\x)=\x^5-\sum_{i=0}^4\widetilde \varphi_i \x^i$ is equal to
    $$
    h_2(\x)
    =
    \x^5-
    \frac{171139 }{5583}\x^4+
    \frac{1759127 }{5583}\x^3-
    \frac{2286645 }{1861}\x^2+
    \frac{2951666 }{1861}\x-
    \frac{750838}{1861}.
    $$
    The candidate for $\widehat \gamma_{8}$, coming from the potential measure for $\widehat \gamma$,
    is
    $$
    \widehat \gamma_{8}
    =
    \sum_{i=1}^{5}\widetilde\varphi_{5-i} \gamma_{8-i}
    =\frac{73385484}{1861}.
    $$
    Already at this point we see that $\widehat \gamma_8\neq \gamma_8$, so the condition of Theorem \ref{the:main-gene} is violated, but let us also compute $M_4$ to see that even the measure for $\widehat \gamma$ does not exist.
    The extended moment matrix $M_4$ is then equal to
    $$
    M_4
    =
    \begin{pmatrix}
        1 & 1 & 2 & 6 & 24  \\[0.2em]
        1 & 2 & 6 & 24 & 120  \\[0.2em]
        2 & 6 & 24 & 120 & 720 \\[0.2em]
        6 & 24 & 120 & 720 & 5040  \\[0.2em]
        24 & 120 & 720 & 5040 &  \frac{73385484}{1861}
    \end{pmatrix}.
    $$
    Its eigenvalues are 
    $40092.4, 86.03, 1.69, 0.25, \mathbf{-0.03}$.
    Since the smallest one is negative, $M_4$ is not positive definite, which is needed for the existence of the measure.
\end{example}

Replacing $x_1=\frac{1}{3}$ with $x_1=1$ in the previous example, there exist a $6$--atomic $\RR$--rm for $\gamma$ with two atoms being $x_1=1$, $x_2=11$, by the following example.

\begin{example}
    \label{ex:measure}
    Let $\gamma\equiv \{\gamma_i\}_{i=0}^9\in \RR^{10}$
    be the same sequence as in Example \ref{ex:no-measure}.
    Let $x_1=1$, $x_2=11$ and $d_1=2, d_2=4$.
    For
    $f(\x)=(\x-1)(\x-11)$ we have
    $$
    \cH_f(3)=
    \begin{pmatrix}
        1 & -7 & -26 & -102 \\[0.2em]
        -7 & -26 & -102 & -456 \\[0.2em]
        -26 & -102 & -456 & -2280 \\[0.2em]
        -102 & -456 & -2280 & -12240 
     \end{pmatrix}.
    $$
    The eigenvalues of $\cH_f(3)$ are 
        $-12683.9, -40.56, 3.28, 0.14$
    and hence $\cH_f(3)$ is invertible. 
    So 
        $$g(\x)=
        \x^4-
        \frac{95824 }{1601}\x^3+
        \frac{753912 }{1601}\x^2-
        \frac{1476768 }{1601}\x
        +\frac{220344}{1601}
        $$
    and $h(\x)=f(\x)g(\x)=\x^6-\sum_{i=0}^5\varphi_i \x^5$ is equal to
    $$
    h(\x)
    =
    \x^6-
    \frac{115036 }{1601}\x^5+
    \frac{1921411 }{1601}\x^4-
    \frac{11577776 }{1601}\x^3+
    \frac{26234592 }{1601}\x^2-
    \frac{18888576}{1601}\x+
    \frac{2423784}{1601}.
    $$
    The candidate for $\gamma_{10}$, coming from the potential measure for $\gamma$,
    is
    $$
    \gamma_{10}
    =
    \sum_{i=1}^{6}\varphi_{6-i} \gamma_{10-i}
    =\frac{5944515264}{1601}.
    $$
    The extended moment matrix $M_5$ is then equal to
    $$
    M_5
    =
    \begin{pmatrix}
        1 & 1 & 2 & 6 & 24 & 120 \\[0.2em]
        1 & 2 & 6 & 24 & 120 & 720 \\[0.2em]
        2 & 6 & 24 & 120 & 720 & 5040 \\[0.2em]
        6 & 24 & 120 & 720 & 5040 & 40320 \\[0.2em]
        24 & 120 & 720 & 5040 & 40320 & 362880 \\[0.2em]
        120 & 720 & 5040 & 40320 & 362880 & \frac{5944515264}{1601}
    \end{pmatrix}.
    $$
    Its eigenvalues are 
    $3.75\cdot 10^6, 5068.64, 38.02, 1.66, 0.35, 0.019$.
    So $M_5$ is positive definite and the measure for $\gamma$ exists. 
    The atoms are $x_1$, $x_2$ and the zeroes of $g(\x)$, i.e., 
    $x_3\approx 0.16, x_4 \approx 2.81, x_5 \approx 5.91, x_6 \approx 50.97$.
\end{example}

\end{document}